\documentclass[11pt]{amsart}
\usepackage{amssymb}
\usepackage{amsmath}
\usepackage{pbox}
\usepackage{array}
\usepackage{bbm}
\usepackage{grcalck}
\usepackage{grcalcx}
\usepackage{grcawol}
\usepackage{comment}
\usepackage{xfrac}

\newtheorem{thm}{Theorem}[section]
\newtheorem{prop}[thm]{Proposition}
\newtheorem{cor}[thm]{Corollary}

\newtheorem{rem}[thm]{Remark}
\newtheorem{defn}[thm]{Definition}

\newtheorem{conj}[thm]{Conjecture}

\newcommand{\C}{{\mathcal C}}
\newcommand{\Z}{{\mathcal Z}}
\newcommand\Rep{\operatorname{Rep}}
\newcommand{\Irr}{\mathcal{I}rr}
\newcommand\Vect{\operatorname{Vec}}

\newcommand\Hom{\operatorname{Hom}}

\newcommand\Id{\operatorname{Id}}

\newcommand\F{\mathbb{F}}

\newcommand\FPdim{\operatorname{FPdim}}

\newcommand\Tr{\operatorname{Tr}}

\newcommand\id{\operatorname{id}}

\begin{document}

\title[Frobenius-Schur indicators for near-group categories]{Frobenius-Schur indicators for near-group and Haagerup-Izumi fusion categories}

\author{Henry Tucker}
\address{Department of Mathematics, University of Southern California, Los Angeles, CA 90089, USA}
\email{htucker@usc.edu}



\date{\today}

\dedicatory{Dedicated to Susan Montgomery}

\begin{abstract}
Ng and Schauenburg generalized higher Frobenius-Schur indicators to pivotal fusion categories and showed that these indicators may be computed utilizing the modular data of the Drinfel'd center of the given category. We consider two classes of fusion categories generated by a single non-invertible simple object: near groups, those fusion categories with one non-invertible object, and Haagerup-Izumi categories, those with one non-invertible object for every invertible object. Examples of both types arise as representations of finite or quantum groups or as Jones standard invariants of finite-depth Murray-von Neumann subfactors. We utilize the computations of the tube algebras due to Izumi and Evans-Gannon to obtain formulae for the Frobenius-Schur indicators of objects in both of these families.
\end{abstract}
 \maketitle



\section{Introduction}

{\it Fusion categories} appear in a wide variety of mathematics and physics. Their objects have the properties of complex representations of finite groups; in particular, they are semisimple and have duals and tensor products. Important examples of fusion categories include categories of representations for Drinfel'd-Jimbo quantum groups and Murray-von Neumann subfactors. From the point of view of these examples fusion categories encode symmetry data in the quantum setting in the same way that finite groups do in the classical setting. Classification problems for these categories do not come without considerable difficulty, therefore it is of great interest to find and understand categorical invariants.

The classical Frobenius-Schur indicator for finite groups was introduced in 1906. It determines if and how a given group representation is self-dual. This was generalized to the setting of semisimple Hopf algebras by Linchenko-Montgomery \cite{lm} and further to the setting of quasi-Hopf algebras \cite{ns-qh} and to pivotal fusion categories \cite{ns-piv} by Ng-Schauenburg.

The FS indicators are a {\it complete invariant} for the Tambara-Yamagami categories \cite{bj}. These are the fusion categories having exactly one non-invertible simple object $\rho$ where $\Hom_\C(\rho\otimes\rho, \rho)=0$. In the present paper we consider the {\bf near-group categories}: those with exactly one non-invertible simple object $\rho$ where $\dim_\mathbb{C}( \Hom_\C(\rho\otimes\rho, \rho))=m$. (The Tambara-Yamagami categories are near-groups with $m=0$.)  We provide the required background on this in Section 2.

Letting $G$ be the group of invertible objects in our near-group category, we find in Section 3 that for the near-group categories with $m=|G|-1$ the indicators are a complete invariant:

\newtheorem*{ng1-cor}{Corollary \ref{ng1-cor}}
\begin{ng1-cor}
The near group categories with $m=|G|-1$ are completely distinguished by their Frobenius-Schur indicators.
\end{ng1-cor}

To make the computations here we utilize \cite[Theorem 4.1]{ns-sph}: the Frobenius-Schur indicators of a spherical fusion category can be computed using the braiding in the Drinfel'd center of the category. A complete list of near-group fusion categories in the case where $m=|G|\leq 13$ was found in \cite{eg}. In each of these examples the modular data for the Drinfel'd centers are given by quadratic forms. From this we get in Section 4:

\newtheorem*{ng2-thm}{Theorem \ref{fs-ng2}}
\begin{ng2-thm}
In all known near group categories with $m=|G|$ the non-invertible object has Frobenius-Schur indicators given by quadratic Gauss sums.
\end{ng2-thm}
This theorem provides new evidence for \cite[Conjecture 2]{eg}: the {\it modular data} (matrix invariants from the braiding) of the centers of these near-groups are always given by quadratic forms. The form of the indicators strongly suggests that these centers are formed from some ``crossed product'' construction for modular categories. See also the ``pasting'' of modular data developed in \cite{eg3}.

Finally, in Section 5, we observe a similar result which supports a similar conjecture for the {\bf Haagerup-Izumi categories}, which are a related family of singly-generated fusion categories having one non-invertible object for {\it each} invertible object:
\newtheorem*{hi-thm}{Theorem \ref{hi-fs}}
\begin{hi-thm}
All known Haagerup-Izumi categories have Frobenius-Schur indicators given by quadratic Gauss sums.
\end{hi-thm}

\subsection{Acknowledgments}

The author wishes to thank his Ph.D. advisor Susan Montgomery; this article comprises the results of the author's dissertation. The author expresses much gratitude to Masaki Izumi, Richard Ng, and Peter Schauenburg for many useful conversations during the development of this work. Thanks are also due to David Penneys for pointing out the paper \cite{eg} to the author and to Vaughan Jones for introducing the author to the the work of Izumi. 

\section{Categorical invariants}

{\bf Tensor categories} are abelian monoidal categories $(\C, \otimes, \mathbbm{1})$ enriched over complex vector spaces; see \cite{egno} or \cite{bk} for the specifics of these definitions. The Mac Lane Strictness Theorem allows us the working assumption of {\bf strictness}: the associativity natural isomorphism is the identity morphism for every triple of objects. Thus we may use diagrammatic notation for the morphisms in these categories. Our notation is read from top to bottom, tensor products are given by side-by-side concatenation, and $\mathbbm{1}$ is not written at all. For examples, the morphisms
\[
\id_V: V\to V,\quad g:V\to U\otimes W, \quad f: \mathbbm{1} \to V_1\otimes \cdots \otimes V_n,
\]
are rendered in diagrammatic notation, respectively, as:
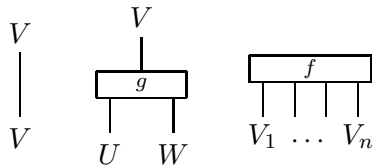
\[
\gbeg14
\got1{V}\gnl
\gcl{2}\gnl
\gob1{V}
\gend
\qquad
\gbeg35
\gvac1\got1{V}\gvac1\gnl 
\gvac1\gcl1\gvac1\gnl
\glmpb\gnot{g}\gcmpt\grmpb\gnl
\gcl1\gvac1\gcl1\gnl 
\gob1{U}\gvac1\gob1{W}
\gend
\qquad
\gbeg44
\gnl
\glmpb\gdnot f\gcmpb\gcmpb\grmpb\gnl
\gcl1\gcl1\gcl1\gcl1\gnl \gob1{V_1}\gob2{\cdots}\gob1{V_n}
\gend
\]

\subsection{Categorifications of semisimple rings}

Tensor categories should be thought of as a {\it categorification} of the notion of a unital algebra. The abelian and monoidal categorical structures are analogues of addition and multiplication, respectively. This point of view asks an obvious question: 
\begin{center}
\framebox{What are the tensor categories that categorify a given ring? }
\end{center}
This question has produced several different interesting classification results for {\bf semisimple} tensor categories: those where every object is a direct sum of some irreducible objects \cite{ty}\cite{iz2}\cite{eg}\cite{eg2}. The set of (isomorphism classes) of the irreducible objects is denoted $\Irr(\C)$. 

Here we consider {\bf fusion categories}: these are semisimple tensor categories $(\C,\otimes,\mathbbm{1})$ that are additionally:
\begin{itemize}
    \item {\bf finitely} semisimple: $|\Irr(\C)|<\infty$ and $\mathbbm{1}\in\Irr(\C)$, and
    \item {\bf rigid}: objects $V\in\C$ have duals $V^*\in\C$ with corresponding maps $ev_V: V^*\otimes V \to \mathbbm{1}$ and $db_V: \mathbbm{1} \to V \otimes V^*$. These are given, respectively, by the diagrams:
    \[
\gbeg33
\got1{V^{*}}\gvac1\got1V\gnl
\gwev3\gnl
\gend
\qquad
\gbeg33
\gnl
\gwdb3\gnl
\gob1 V \gvac1 \gob1{V^{*}}
\gend
\]
satisfying the relations:
\[
        \gbeg64
        \gvac4\got{1}{V} \gnl
        \gwdb3 \gvac1 \gcl1 \gnl
        \gcl1 \gvac1 \gwev3 \gnl
        \gob{1}{V}
        \gend =
        \gbeg14
        \got{1}{V}\gnl
        \gcl2\gnl
        \gob{1}{V}
        \gend
        \qquad \textrm{and} \qquad
        \gbeg64
        \got{1}{V^{*}} \gvac4 \gnl
        \gcl1 \gvac1 \gwdb3 \gnl
        \gwev3 \gvac1 \gcl1 \gnl
        \gvac4 \gob{1}{V^{*}}
        \gend =
        \gbeg14
        \got{1}{V^{*}}\gnl
        \gcl2\gnl
        \gob{1}{V^{*}}
        \gend
\]

\end{itemize} 

These two requirements are meant to make the objects of $\C$ behave like group representations. Indeed, the tensor category $\Rep(G)$ of complex representations of a finite group $G$ is the prototypical example of a fusion category: Maschke's theorem gives finite semisimplicity and the contragredient representation gives rigidity.

Now we make precise the notion of categorification. The {\bf Grothendieck ring} $K_0(\C)$ of a fusion category $\C$ is the $\mathbb{Z}$-based ring with basis $\Irr(\C)$, multiplication given by the tensor product in $\C$, and addition given by the direct sum in $\C$; that is, $K_0(\C)$ captures the ring structure of the category and forgets the morphisms. In the example of $\Rep(G)$ it is the character ring $R(G)$. We say that $\C$ {\bf categorifies} a ring $K$ if $K_0(\C)=K$.

The simplest class of based rings to consider are the group rings $\mathbb{Z}G$, which are categorified precisely by the {\bf pointed} fusion categories. These are the categories $\Vect_G^\omega$ of $G$-graded vector spaces where the associativity morphism for the tensor product of three irreducible objects is given by a 3-cocycle $\omega\in Z^3(G,\mathbb{C}^\times)$. These categories are classified to equivalence by the cohomology class of $[\omega]\in H^3(G,\mathbb{C}^\times)$. These facts are due to Mac Lane, cf. \cite[Prop. 4.10.3]{egno}.

Here we will consider another level of complication; based rings with (only) one non-invertible object:

\begin{defn}
Let $G$ be a finite group. A fusion category $\C$ is a {\bf near group} if its Grothendieck ring is given by: 
\[
K_0(\C) = NG(G,m) := \mathbb{Z}[G \cup \{ \rho \}]
\]
where multiplication is given by the group law and, where $g\in G$:
\[
\rho g = \rho = g \rho \quad \textrm{and} \quad \rho^2 = m\rho + \sum\limits_{h\in G} h
\]
\end{defn}

\begin{rem}\cite[Theorem 2(a)]{eg}
When $G$ is abelian the multiplicity $m$ is restricted to the following values:
\begin{itemize}
\item $m=|G|-1$, or
\item $m=k|G|$ for some $k\in \mathbb{N}$
\end{itemize}
\end{rem}

\noindent Consider the following important examples.
\begin{enumerate}
\item $\Rep(\mathcal{D}_8)$ and $\Rep(Q_8)$ categorify $NG(\mathbb{Z}/(2) \times \mathbb{Z}/(2),0)$. These are examples of {\bf Tambara-Yamagami} categories, which are the near groups with $m=0$ \cite{ty}.
\item $\Rep(S_3)$ and $\Rep(A_4)$ categorify $NG(\mathbb{Z}/(2),1)$ and $NG(\mathbb{Z}/(3),2)$, respectively. These are examples where $m=|G|-1$.
\item The principal even sectors of the $D_5$ Murray-von Neumann subfactor also categorify $NG(\mathbb{Z}/(2),1)$.
\item $\Rep(AGL_1(\F_q))$ categorifies $NG(\mathbb{Z}/(q-1), q-2)$.
\item The principal even sectors of the $A_4$, $E_6$, and Izumi-Xu subfactors categorify $NG(\mathbb{Z}/(1),1)$, $NG(\mathbb{Z}/(2), 2)$, and $NG(\mathbb{Z}/(3), 3)$, respectively.
\end{enumerate}

\subsection{Frobenius-Schur indicators}

It is known that the Grothendieck ring and the associativity natural isomorphism completely determine a fusion category up to monoidal equivalence \cite[\S 4.9-4.10]{egno}. The associativity data is encoded by the {\bf 6j symbols}, which are the matrix components of the linear maps induced by the associativity natural isomorphism on triples of simple objects. Directly classifying all 6j symbols having a given Grothendieck ring is difficult in general as it requires finding solutions to large systems of non-linear equations.

The near groups are {\bf spherical} fusion categories. These are the fusion categories equipped with a natural isomorphism $j: \; \Id_\C \xrightarrow{\sim} (\cdot)^{**}$ (that is, a {\bf pivotal structure)} such that
the left and right quantum traces agree for all $V \in \Irr(\C)$:

\[
qtr^r(f) := \quad
\gbeg45
\gwdb3\gnl
\gnot{f}\gvac2\gcl3\gnl
\gcl1\gvac2\got1{V^{*}}\gnl
\gnot{j_V}\gnl
\gwev3
\gend
\quad = \quad
\gbeg45
\gvac1\gwdb3\gnl
\gvac1\gcl3\gvac1\gnot{f}\gnl
\got1{V^*}\gvac2\gcl1\gnl
\gvac3\gnot{j^{-1}_V}\gnl
\gvac1\gwev3
\gend
 \quad =: qtr^l(f)
\]

These categorify unital algebras {\it with involution}. The spherical structure defines a categorical dimension of an object $V\in \C$ by $dim(V)=qtr(\id_V)$. 

For $\C$ a spherical fusion category we can define a finer categorical invariant than the Grothendieck ring:

\begin{defn} \cite{ns-piv}
 For $V \in \C$ we define the {\bf $k^{\text{th}}$ Frobenius-Schur indicator} by the following linear trace:
\[
\nu_k(V)=\Tr\left(E^{(k)}_V: \underbrace{\gbeg43\glmpb\gdnot f\gcmpb\gcmpb\grmpb\gnl
  \gcl1\gcl1\gnot{\cdots}\gvac1\gcl1\gnl
  \gob1V\gob1V\gob1\cdots\gob1V\gend}_n\mapsto
  \gbeg65\gwdb6\gnl
  \gcl1\glmpb\gdnot f\gcmpb\gcmpb\grmpb\gcl1\gnl
  \gev\gcl2\gnot\cdots\gvac1\gcl2\gnot{j^{-1}_V}\gnl
  \gvac5 \gcl1\gnl
  \gvac2\gob1V\gob1\cdots\gob1V\gob1V\gend
  \right)
 \]
where $E_V^{(k)}$ is a linear endomorphism of the finite-dimensional vector space $\Hom(\,\mathbbm{1},V^{\otimes n})$ taking $V^{\otimes n}$ to be $n$-fold tensor product of $V$ with all parentheses to the right.
\end{defn}

The Tambara-Yamagami categories are an example of a fusion category family where the Frobenius-Schur indicators are a finer invariant than the Grothendieck ring \cite{ns-qh}. In \cite{bj} it was shown that the indicators are a {\it complete} invariant for the Tambara-Yamagami categories. That is, the monoidal equivalence classes of fusion categories associated to the ring $NG(G,0)$ are completely distinguished by their Frobenius-Schur indicators. We give this property a name:

\begin{defn}
A ring $K$ exhibits {\bf FS indicator rigidity} if its categorifications can all be distinguished by their Frobenius-Schur indicators.
\end{defn}
\noindent  The central question that motivates the present article is immediate:
\begin{center}
\framebox{What rings have FS indicator rigidity?}
\end{center}
We will see in Corollary \ref{ng1-cor} that the near-group rings $NG(G,|G|-1)$ exhibit this property.

\subsection{Drinfel'd centers and modular data} The Drinfel'd center $\Z(\C)$ of a spherical fusion category $\C$ is {\bf modular} \cite[Prop. 5.10]{mu}, that is it is again spherical with a {\bf non-degenerate braiding} $c_{V,W}:V\otimes W \to V \otimes W$ which is given in diagrams by:
\[
  c_{V,W}=\gbeg23\got1V\got1W\gnl\gbr\gnl\gob1W\gob1V\gend
\]
Modular categories come with a projective representation of the modular group called {\bf modular data}. The representation is defined by sending the generators $\mathfrak{s}, \mathfrak{t} \in SL_2(\mathbb{Z})$ to the {\bf $S$- and $T$-matrices}:

\[
S =
\left(
 \gbeg68
\gvac1\gwdb5\gnl
\gvac1\gcl2\gwdb3\gcl6\gnl
\got1V\gvac1\gcl1\got1W\gcl4\gnl
\gvac1\gbr\gvac1\gnl
\gvac1\gbr\gnl
\gvac1\gcl2\gcl1\gnl
\gvac2\gwev3\gnl
\gvac1\gwev5\gnl
\gend 
\right)_{V, W \in \Irr(\C)}
\qquad
T = \left( \theta_V = \delta_{V,W}
\gbeg37
  \gvac2\got1V\gnl
  \gdb\gcl1 \gnl
  \gcl1\gibr\gnl
  \gev\gcl1\gnl
  \gvac2 \gnot{j_V^{-1}}\gnl
  \gvac2 \gcl1\gnl
  \gvac2 \gob1V
  \gend
  \right)_{V,W \in \Irr(\C)}
\]
Most crucially, we can obtain the Frobenius-Schur indicators from the modular data of the Drinfel'd center:

\begin{thm}\cite[Theorem 4.1]{ns-sph}\label{fs-sph}
Let $\C$ be a spherical fusion category and let $F: \Z(\C) \to \C$ be the forgetful functor. Then:
\[
\nu_k(X) = \frac{1}{qdim(\C)} \sum_{V \in \Irr(\Z(\C))} \theta_V^k qdim(V) dim_\mathbb{C}(\Hom_\C(F(V),X))
\]
where $\theta_V$ are the entries of the $T$-matrix for $\mathcal{Z}(\C)$.
\end{thm}

\section{Frobenius-Schur indicators for near groups with $m=|G|-1$}

Let $\C$ be a fusion category such that $K_0(\C)=NG(G,|G|-1)$. It is shown in \cite[Proposition 2]{eg} that such a fusion category can only exist if $G \cong \F_{|G|+1}^\times$ is the multiplication group of a finite field. (So $G$ is cyclic, and thus $H^2(G, \mathbb{T})=1$.) Let $p = \mathrm{char}(\F_{|G|+1})$.

Consider again the category $\Rep(AGL_1(\F_{|G|+1}))$. These provide the main examples of $m=|G|-1$ near groups. In fact, by \cite[Corollary 7.4]{ego} and \cite[Proposition 5]{eg}, these are the {\it only} fusion categories with this Grothendieck ring unless $|G|=1,2,3,7$.

\subsection{Indicators for $\C \simeq \Rep(AGL_1(\F_q))$}

We may use classical methods to determine the indicators for $\C$ that is tensor equivalent to the category of representations of an affine general linear group of degree 1 over the finite field $\F_q$. Recall that $\theta^G_k(h) = |\{g \in G \, | \, g^k = h\}|$.

\begin{prop}
Suppose $\C$ is such that $K_0(\C)=NG(G,|G|-1)$ and $|G| \neq 1, 2, 3, 7$. Then $\C \simeq_\otimes \Rep(AGL_1(\F_{|G|+1}))$ and:
\[
\nu_k(\rho) = \theta_k^G(e) - 1 + \delta_{\lfloor \frac{k}{p} \rfloor, \frac{k}{p}}
\]
\end{prop}

\begin{proof}
Let $|G| + 1 = q$. Since $AGL_1(\F_q) \cong \F_q \rtimes \F_q^\times$ we may use Serre's method of little groups \cite[\S 8.2, Proposition 25]{serre} to see that the character $\rho$ for the irreducible representation with degree $>1$ is given by:
\[
\rho(a, b) =  \frac{\delta_{1,b}}{q} \sum\limits_{(x,y) \in AGL_1(\F_q)}  \eta(y^{-1}a)
\]
for any {\it non-trivial} linear character $\eta \in \widehat{(\F_q, +)}$.

Now we may apply the classical formula for $\nu_k(\rho)$ \cite[Lemma 4.4]{isaacs}:
\begin{align*}
\nu_k(\rho) &=
\frac{1}{q(q-1)} \sum\limits_{(a,b) \in \F_q \rtimes \F_q^\times} \rho((a,b)^k)\\
&= \frac{1}{q(q-1)} \sum\limits_{(a,b),\; b^k = 1} \rho((1+b+b^2+\cdots+b^k)a, 1)\\
&= \frac{1}{q(q-1)} \left( \sum\limits_{(a,b), \; b^k=1, \; b\neq 1} \rho(0,1) + \sum\limits_{n \in \F_q} \rho(kn, 1) \right)\\
\end{align*}
Since $\rho$ is a degree $q-1$ character, the left hand sum is $q(q-1)(\theta^{\F_q^\times}_k(1) - 1)$. The right hand sum is equal to:
\begin{equation*}
\sum\limits_{n\in\F_q} \sum\limits_{b\in \F_q^\times} \eta(b^{-1}kn) =
\begin{cases}
    q(q-1) & \text{if } p\,|\,k\\ 0 & \text{if } p\nmid k
\end{cases}
\end{equation*}
The $p\,|\,k$ case is clear since then $\eta(b^{-1}kn)$ is identically 1. On the other hand, $\eta(b^{-1}kn) = b\cdot \eta(kn)$ under the transpose of the left regular action of $\F_q^\times \cong GL_1(\F_q)$ on $\widehat{\F_q}\cong\F_q$. Since $(p, k)=1$ we have that: 
\[
\sum\limits_{n \in \F_q} b\cdot \eta(kn) = \sum\limits_{n \in \F_q} b\cdot \eta(n)
\]
and since the action is faithful by definition, we know that $b\cdot\eta$ is not the trivial representation for any $b \in \F_q^\times$. Hence by orthogonality of characters the sum is 0. The formula is now clear since the given Kronecker delta is 1 if $p\,|\,k$ and is 0 otherwise.
\end{proof}

\subsection{Indicators in general from modular data of $\mathcal{Z}(\C)$}
For $|G|=1, 3, 7$ there is 1 additional monoidal equivalence class, and for $|G| = 2$ there are 2 additional monoidal equivalence classes. The modular data for Drinfel'd centers of $m=|G|-1$ near groups was computed in \cite[Theorem 5]{eg}. We will appeal to Theorem \ref{fs-sph} to compute the indicators for a general $m=|G|-1$ near group.

Let $\epsilon \in \widehat{G}$ be the trivial character and let $\F_{|G|+1}^{+}$ be the additative group of the finite field. Excluding the case where $|G|=7$ and $s=-1$ we have the following data for the center $\Z(\C)$:

\begin{center}
    \setlength{\arrayrulewidth}{0.5mm}
    \setlength{\tabcolsep}{12pt}
    \renewcommand{\arraystretch}{2}
    \begin{tabular}{ | c | c | c | c |}
    \hline\hline
    $X \in \Irr(\Z(\C))$ &  $F(X)$  &  $c_{X,-}$ given by  &  $\theta_X$ \\ \hline\hline
    $A_g \quad (g\in G)$ & $g$ & 1 & 1 \\ \hline
    $\Sigma$ & $\displaystyle \bigoplus_{x \in G} x$ & 1 & 1 \\ \hline
    $B^\omega_g \quad (g \in G)$ & $\rho + g$ & $\omega \in \widehat{G}\backslash\{\epsilon\}$ & $\overline{\omega(g)}$ \\ \hline
    $C^\psi \quad \left(\psi \in \widehat{\F_{|G|+1}^{+}}\right)$ & $\rho$ & $\psi \in \widehat{\F_{|G|+1}^{+}}$ & $\overline{\zeta_1 \psi(1)}$ \\
    \hline
    \end{tabular}
\end{center}
where the half-braiding for $C^\psi$ on occurrences of $\rho$ in objects of $\Z(\C)$ is a morphism:
\[
e_{C^\psi}(\rho) \in \Hom_\C(\rho\otimes\rho, \rho\otimes\rho) \cong \mathbb{C}^{|G|} \oplus M_m(\mathbb{C})
\]
given by:
\[
e_{\C^\psi}(\rho) = \zeta_1\psi(1)\left(\bigoplus\limits_{k \in G} (-1)^{mk}\Id_k\right) \; \bigoplus \; \left[ \zeta_\gamma (\psi\circ\sigma)(\gamma) \delta_{\sigma^2(\gamma)^*, \mu} \Id_\rho \right]_{\gamma, \mu}
\]

For the case where $|G| = 7$ and $s = -1$ we have:

\begin{center}
    \setlength{\arrayrulewidth}{0.5mm}
    \setlength{\tabcolsep}{12pt}
    \renewcommand{\arraystretch}{2}
    \begin{tabular}{ | c | c | c | c |}
    \hline\hline
    $X \in \Irr(\Z(\C))$ &  $F(X)$  & $c_{X,-}$ given by  & $\theta_X$ \\ \hline\hline
    $A_g \quad (g\in G)$ & $g$ & 1 & 1 \\ \hline
    $\Sigma$ & $\displaystyle \bigoplus_{x \in G} x$ & 1 & 1 \\ \hline
    $B^\omega_g \quad (g\in G,\, \omega \in \widehat{G}\backslash\{\epsilon\})$ & $\rho + g$ & $\omega \in \widehat{G}\backslash\{\epsilon\}$ & $\overline{\omega(g)}$ \\ \hline
    $E_1$ & $\rho + \rho$ & 1 & $i$ \\ \hline
    $E_2$ & $\rho + \rho$ & 1 & $-i$\\
    \hline
    \end{tabular}
\end{center}

\noindent With the preceding data in hand we may now apply Theorem \ref{fs-sph} to see:

\begin{thm}\label{fs-ng1}
Suppose $\C$ is such that $K_0(C)=NG(G, |G|-1)$. Then the indicators for the non-invertible object $\rho$ are given by:
\begin{enumerate}
\item If $|G|\neq7$ or $s=1$ then: 
\[
\nu_k(\rho) = (\theta_k^G(e) - 1) + \overline{\zeta_1}^k \delta_{\lfloor \frac{k}{p} \rfloor, \frac{k}{p}}
\]
\item If $|G|=7$ and $s=-1$ then: 
\[
\nu_k(\rho) = (\theta_k^G(e) - 1) + (-1)^{k/2}\delta_{\lfloor \frac{k}{2} \rfloor, \frac{k}{2}}
\]
\end{enumerate}
\end{thm}

\begin{proof}
(1) Suppose $|G|\neq 7$ or $s=1$. Then we have:
\begin{align*}
\nu_k(\rho) &= \frac{1}{\FPdim{\C}} \left( \sum_{\substack{g \in G\\ \omega \in \widehat{G}\backslash\{\epsilon\}}} \theta_{B_g^\omega}^k\FPdim(B_g^\omega) + \sum_{\psi \in \widehat{\F_{|G|+1}^{+}}} \theta_{C^\theta}^k \FPdim(C^\psi) \right)\\
 &= \frac{1}{\FPdim(\C)} \left( (|G|+1) \sum_{\substack{g \in G\\ \omega \in \widehat{G}\backslash\{\epsilon\}}} \overline{\omega(g)}^k + |G|\overline{\zeta_1}^k \sum_{\psi \in \widehat{\F_{|G|+1}^{+}}} \overline{\psi(1)}^k \right)
\end{align*}

Consider the first summand. Since $G$ is abelian we may choose an isomorphism $h \mapsto \chi_h$ from $G \to \widehat{G}$. Then we have:

\begin{align*}
\sum_{\substack{g \in G\\ \omega \in \widehat{G}\backslash\{\epsilon\}}} \overline{\omega(g)}^k &= \left( \sum_{g \in G}\sum_{\omega \in \widehat{G}} \overline{\omega(g)}^k \right) - \left(\sum_{g \in G}\overline{\epsilon(g)}^k \right)\\
&= \left( \sum_{g \in G} \sum_{h \in G} \overline{\chi_h(g)}^k \right) - |G|\\
&= \left( |G|\sum_{h \in G} \overline{\nu_k(\chi_h)} \right) - |G|\\
&= |G|( \theta^G_k(e) - 1)
\end{align*}

Consider the second summand. Since $\F_{n+1}^+$ is the additive group of a finite field we have that $n + 1 = p^l$ for some prime $p$ and positive integer $l$ and that $\F_{n+1}^+ \cong (\mathbb{Z}_p)^l$ as groups. Under this identification the multiplicative unit $1 \in \F_{n+1}^+$ is a direct sum of generators of the copies of $\mathbb{Z}_p$.

\begin{align*}
\sum_{\psi \in \widehat{\F_{|G|+1}^{+}}} \overline{\psi(1)}^k = \sum_{\psi \in \widehat{\F_{|G|+1}^{+}}} \overline{\psi(k1)} &= \bigg\{
     \begin{array}{lr}
       0 & \text{if } k1 \neq 0\\
       p^l & \text{if } k1 = 0
     \end{array}\\
   &= \bigg\{
     \begin{array}{lr}
       0 & \text{if } p \nmid k\\
       |G|+1 & \text{if } p | k
     \end{array}\\
   &\\     
   &= \left(|G| + 1\right)\delta_{ \lfloor\frac{k}{p}\rfloor, \frac{k}{p} }
\end{align*}

(2) Now suppose that $|G|=7$ and $s=-1$. Then:
\begin{align*}
\nu_k(\rho) &= \frac{1}{\FPdim{\C}} \left( \sum_{\substack{g \in G\\ \omega \in \widehat{G}\backslash\{\epsilon\}}} \theta_{B_g^\omega}^k\FPdim(B_g^\omega) + 2\sum_{i=1}^2 \theta_{E_t}^k \FPdim(E_t) \right)\\
 &= \theta_k^G(e) - 1 + \frac{4|G|i^k(1 + (-1)^k)}{|G|+|G|^2}\\
 &= \theta_k^G(e) - 1 + \frac{i^k(1+(-1)^k)}{2}\\
 &= \theta_k^G(e) - 1 + (-1)^{k/2}\delta_{\lfloor \frac{k}{2} \rfloor, \frac{k}{2}}
\end{align*}
\end{proof}

\begin{cor}\label{ng1-cor}
The near group fusion ring $NG(G,|G|-1)$ exhibits Frobenius-Schur indicator rigidity.
\end{cor}
\begin{proof}
The statement is vacuous in all but the cases where $|G|=1,2,3,7$. We shall consider them now:\\
\indent If $|G|=1,3,7$ then there is one additional tensor equivalence class corresponding to $s=-1$. By \cite[p.41]{eg} if $|G| + 1$ is even (i.e. a power of 2) then $\zeta_1^2=s$, hence $\nu_2(\rho) = s$ in each of these three cases.\\
\indent If $|G|=2$ then $s=1$ but instead $b = \mu$ where $\mu$ is some third root of unity. The two non-trivial possibilities for $\mu$ correspond to the two additional tensor equivalence classes for this type. By \cite[p.42]{eg} if $\mu = e^{\pm \frac{2\pi i}{3}}$ then $\zeta_1 = e^{\mp \frac{2\pi i}{3}}$, hence $\nu_3(\rho)=\mu$.
\end{proof}


\section{Frobenius-Schur indicators for near groups with $m=|G|$}

For the rest of this article the group operation in $G$ will be written {\it additively}. This will be a more convenient notation for working with bilinear and quadratic forms.

\subsection{Metric groups and the Fourier transform}

Shimizu observed that the Fourier transform for finite groups appears when computing Frobenius-Shur indicators for fusion categories \cite{shim}. A finite abelian group $G$ is isomorphic to its linear dual $\widehat{G}$ via a non-degenerate symmetric {\bf bicharacter} $\langle, \rangle$ with the following identification:
\begin{align*}
    &G \to \widehat{G}\\
    &g \mapsto \langle g, \cdot \rangle
\end{align*}
Symmetric bicharacters $\langle,\rangle: G \times G \to \mathbb{T}$ are in one-to-one correspondence with bilinear forms $\beta: G \times G \to \mathbb{Q}/\mathbb{Z}$ via the exponential:
\[
\langle g, h \rangle = e^{2\pi i \beta(g,h)}
\]

\noindent A {\bf quadratic form} is a function $q: G \to \mathbb{Q}/\mathbb{Z}$ with $q(-g) = q(g)$ such that
\[
\partial q (g, h) := q(g) + q(h) - q(gh)
\]
is a symmetric bilinear form. A pair $(G, q)$ is called a {\bf pre-metric group}. If the bilinear form $\partial q$ is non-degenerate then it is called a {\bf metric group}. 

\begin{rem}
If $|G|$ is odd then the correspondence between quadratic forms and bilinear forms given by $\partial$ is one-to-one. If $|G|$ is even then the correspondence is $|G/2G|$-to-one.
\end{rem}

Now, using the bicharacter $\langle, \rangle$ we can define the {\bf Fourier transform} for complex function $f: G \to \mathbb{C}$ on finite abelian groups:
\[
\widehat{f}(g) = \frac{1}{\sqrt{|G|}}\sum_{h\in G} \overline{\langle g, h \rangle}f(h)
\]
The Fourier transform of the exponent of a quadratic form $q$ at the unit element of the group defines a very important invariant of pre-metric groups:
\begin{defn}
Let $(G, q)$ be a pre-metric group. Then the Fourier transform of $e^{2\pi i q}$ at $0\in G$ defines the {\bf Gauss sum}:
\[
\Theta(G,q) = \frac{1}{\sqrt{|G|}} \sum\limits_{g \in G} e^{2\pi i q(g)}
\]
\end{defn}

The Gauss sum is multiplicative over the (obviously defined) orthogonal direct product of metric groups:
\[
\Theta(G\perp G', q+q') = \Theta(G,q)\Theta(G',q')
\]
(This identity is for {\it metric} groups, hence the quadratic forms must all be non-degenerate.)

\subsection{Izumi's classification of $m=|G|$ near groups}

The classification of the Tambara-Yamagami categories was completed by direct solution of the equations resulting from the pentagon axiom for the associativity. This method is not feasible for more complicated categories.

Izumi was able to extend this classification by using a fundamental result due to Popa: every unitary fusion category tensor-generated by one object can be embedded in the category of sectors of the hyperfinite type III Murray-von Neumann subfactor $R$. Sectors are unitary equivalence classes of endomorphisms of $R$; the tensor product of sectors is composition. Note that the 6j symbols can be obtained from Izumi's classification data. See \cite{sw} for an example for the near-group category $\C$ with $K_0(\C)=NG(\mathbb{Z}/(3), 3)$ coming from the $E_6$ subfactor.

Izumi completely classified {\bf unitary} (or $\mathbf{C^*}$) near-group fusion categories where $H^2(G,\mathbb{C}^\times)=1$ in \cite{iz1,iz2,iz3}.

\begin{thm}\cite[Theorem 5.3]{iz2}\label{iz-ng-class}
Unitary fusion categories $\C$ such that $K_0(\C)=NG(G,|G|)$ and $H^2(G,\mathbb{C}^\times)=1$ are classified up to monoidal equivalence by the group $G$, a metric group structure $\langle,\rangle$ on $G$, and the following complex parameters:
\begin{enumerate}
\item $a: G \to \mathbb{T}$ such that:

$a(g)=e^{2\pi i q(g)}$ for a quadratic form  $q$ with $\langle g, h \rangle = e^{2\pi i \partial q(g,h)}$, i.e.:
\[
a(0)=1, \quad a(g)=a(-g), \quad \frac{a(g+h)}{a(g)a(h)} = \langle g,h \rangle
\]
\item $b: G \to \mathbb{C}$ and $c\in\mathbb{T}$ such that:
\[
\Theta(G,q) = \widehat{a}(0) = \frac{1}{c^3}
\]
\begin{gather*}
b(g) = \overline{a(g)b(-g)}\\
\widehat{b}(0)=\frac{-c}{\dim(\rho)}\\
\widehat{b}(g)=c\overline{b(g)} \qquad |\widehat{b}(g)|^2 = \frac{1}{|G|} - \frac{\delta_{g,0}}{\dim(X)}\\
\sum_{x\in G} b(x+g)b(x+h)\overline{b(x)} = \langle g, h \rangle b(g)b(h) - \frac{c}{\dim(\rho)\sqrt{|G|}}
\end{gather*}

\end{enumerate}
Two such fusion categories $\mathcal{NG}(G_1,\langle,\rangle_1,a_1, b_1, c_1)$ and $\mathcal{NG}(G_2,\langle,\rangle_2,a_2, b_2, c_2)$ 
are monoidally equivalent if and only if 
\[
c_1 = c_2
\]
and there is an isomorphism of metric groups $\phi:(G_1,\langle,\rangle_1)\to(G_2, \langle,\rangle_2)$ such that:
\[
a_1 = a_1\circ\phi \quad \textrm{and} \quad b_2 = b_1\circ\phi
\]
\end{thm} 

\begin{rem}
If the requirement that $H^2(G,\mathbb{C}^\times)=1$ is relaxed then a solution of Izumi's equations in Theorem \ref{iz-ng-class} is sufficient to produce a near-group category with the required $K_0$ ring, but \emph{not} necessary.
\end{rem}

\subsection{Indicators from modular data of $\mathcal{Z}(\C)$}\label{j-mod-data}

Izumi found the simple objects of $\mathcal{Z}(\C)$ along with their twists and half-braidings in \cite[Theorem 6.8]{iz2} which is given as follows:

\begin{center}
    \setlength{\arrayrulewidth}{0.5mm}
    \setlength{\tabcolsep}{12pt}
    \renewcommand{\arraystretch}{2}
    \begin{tabular}{ | c | c | c | c |}
    \hline\hline
    $X \in \Irr(\Z(\C))$ & $F(X)$ &  $c_{X,-}$ given by &  $\theta_X$\\ \hline\hline
    $A_g \quad (g \in G)$ & $g$ & 1 & $\langle g, g \rangle$ \\ \hline
    $B_g \quad (g\in G)$ & $\rho + g$ & 1 & $\langle g, g \rangle$ \\ \hline
    $C_{g, h} \quad (g\neq h \in G)$ & $\rho + g + h$ & 1 &  $\langle g, h \rangle$ \\
    \hline
    $E_j$ for $j = 1, ..., \frac{|G|(|G|+3)}{2}$ & $\rho$ &  & $\omega_j$ \\ \hline
    \end{tabular}
\end{center}

The $\omega_j \in \mu_\infty \subseteq \mathbb{T}$ are solutions to the system of equations (6.18)-(6.20) in \cite[\S 6]{iz2} parametrized by $g \in G$ with coefficients given by the complex values $a(g), b(g), c \in \mathbb{C}$. 

\begin{prop}\label{fs-omega-form}
Suppose $\C$ is fusion category with $K_0(\C)=NG(G, |G|)$ with non-invertible object $\rho$. Let $q$ be a quadratic form such that $\langle g,h \rangle = e^{2\pi i \partial q(g,h)}$. Then the indicators for $\rho$ are given by the following:
\[
\nu_k(\rho) = \frac{1}{2}\theta^G_k(e) + \frac{qdim(\rho)}{qdim(\C)}\left(\frac{\sqrt{|G|}}{2}\Theta(G, 2kq) + \sum\limits_{j=1}^{|G|(|G|+3)/2} \omega_j^k \right)
\]
\end{prop}

\begin{proof}
Let $d_\rho := qdim(\rho)$ and let $<$ be an arbitary ordering on the finite group $G$. Again applying Theorem \ref{fs-sph} we have the following:
\begin{align*}
\nu_k(\rho) &= \frac{1}{qdim(\C)} \left( (1+d_\rho) \sum\limits_{g \in G} \theta_{B_g}^k + (2 + d_\rho)\sum\limits_{\substack{g,h \in G\\ g<h}} \theta_{C_{g,h}}^k + d_\rho \sum\limits_{j=1}^{|G|(|G|+3)/2} \theta_{E_j}^k \right)\\
&= \frac{1}{qdim(\C)} \left( \frac{d_\rho}{2} \sum\limits_{g \in G} \langle g,g \rangle^k + \frac{2+d_\rho}{2} \sum\limits_{g,h \in G} \langle g,h \rangle^k + d_\rho \sum\limits_{j} \omega_j^k \right)
\end{align*}
where the second equality is due to the symmetry of $\langle,\rangle$.\\
\indent Now we consider the middle sum:
\begin{align*}
\sum\limits_{g,h \in G} \langle g,h\rangle^k = \sum\limits_{g\in G} \sum\limits_{h\in G} \langle g, h^k \rangle = |G|\sum\limits_{g \in G} \nu^{\text{groups}}_k(\langle g,\cdot \rangle) = |G|\theta^G_k(e)
\end{align*}
The second equality is by definition of the Frobenius-Schur indicator for finite groups (denoted $\nu^{\text{groups}}_k$) \cite[Equation 4.4]{isaacs} and the third equality is by \cite[p. 49]{isaacs}.\\
\indent Now let $q$ be a quadratic form such that $\langle g, h \rangle = e^{2\pi i \partial q(g,h)}$ and consider the first sum:
\begin{align*}
\sum\limits_{g \in G} \langle g,g \rangle^k = \sum\limits_{g \in G} e^{2\pi i (2kq(g))} = \sqrt{|G|}\Theta(G, 2kq)
\end{align*}
Hence the formula is now clear.
\end{proof}

\subsection{Modular data for pointed modular categories}

Recall that any {\it pointed} fusion category is equivalent to $\Vect_G^\omega$ for some $[\omega]\in H^3(G,\mathbb{C}^\times)$. Now we consider pointed {\it modular} categories. Since modular categories are also fusion categories they will be equivalent as fusion categories to $\Vect_G^\omega$ with $G$ abelian. The braiding induces a quadratic form $c_{g,g} = e^{2\pi i q(g)}$, which gives $G$ the structure of a metric group. Then these categories are classified under {\it braided} equivalence up to isomorphism of {\it pre-metric} groups. Note that in the case of {\it odd-order} groups if $\omega$ admits a braiding it will be unique; the notational convention $\Vect^{(\omega, c)}_G$ includes the braiding $c$. (The preceding is the result of the non-trivial fact that the Eilenburg--Mac Lane abelian cohomology $H^3_{ab}(G,\mathbb{C}^\times)$ is isomorphic to the group of quadratic forms on $G$. See \cite[\S 8.4]{egno} for details.)

We now give the modular data for a pointed modular category. Define the bicharacter $\langle g,h \rangle_q := e^{2\pi i \partial q(g,h)}$. The modular data are given by the {\bf Weil representation} associated to the pre-metric group $(G,q)$:
\[
S = S^q := \frac{1}{\sqrt{|G|}}(\overline{\langle g,h \rangle}_q)_{g,h\in G} \qquad T = T^q := (\delta_{g,h} e^{2\pi i q(g)})_{g,h\in G}
\]

\subsection{Indicators when $|G|$ is odd}

When $|G|$ is odd we have a one-to-one between quadratic forms and bilinear forms given by the map $q \mapsto \partial q$. Let $q$ be the quadratic form on $G$ such that $\langle g,h \rangle = e^{2\pi i \partial q(g,h)}$. Then we define 
\[
\mathcal{NG}(G,q,b,c):= \mathcal{NG}(G, \langle,\rangle_q,e^{2\pi i q}, b, c)
\]
is the corresponding near-group fusion category via the notation from Theorem \ref{iz-ng-class}. Evans-Gannon conjecture the following:

\begin{conj}\label{ng-conj}\cite[Conjecture 2]{eg}
Suppose $|G|$ is odd. Then there exists a metric group $(G',q')$ of order $|G|+4$ such that: 
\begin{enumerate}
\item Simple objects $E_j$ in Theorem \ref{j-mod-data} are indexed by $g \in G, x \in G'\backslash \{ e \}$ where $E_{g, x} = E_{g, x^{-1}}$ and:
\[
\theta_{E_{g, x}} = \langle g,g \rangle e^{2\pi i \partial q'(x)}
\]
\item The modular data are given by the Kronecker product of the Weil representation for $(G,q)$ with another pair of modular data $(S',T')$ for a rank $|G|+3$ modular category:
\[
S^{q,q'} := S^q \otimes S' \qquad T^{q,q'} : = T^q \otimes T'
\]
where we have 
\[
T' = Diag(1, 1, \langle g,g \rangle_q, \langle x, x \rangle_{q'})_{g\in G, x\in G'}
\]
\end{enumerate}
\end{conj}
(See \cite[Proposition 7]{eg} for the definition of $S'$.)

\begin{rem}
In \cite{eg} they show that the conjecture is true for near groups with $|G|\leq 13$ odd.
\end{rem}

\begin{thm}\label{fs-ng2}
Suppose a fusion category $\C$ with $K_0(\C)=NG(G,|G|)$ and $|G|$ odd satisfies Conjecture \ref{ng-conj}. Then:
\[
\nu_k(\rho) = \frac{1}{2}\theta^G_k(e) + \frac{1}{2} \Theta(G, 2kq)\Theta(G',2kq')
\]
\end{thm}
\begin{proof}
Let $N = (|G'|-1)/2$ and enumerate $G'$ as follows:
\[
G' = \{e, x_1, \ldots, x_N, x_1^{-1}, \ldots, x_N^{-1} \}
\]
Let $d_\rho := qdim(\rho)$ and $\langle x,y\rangle_{q'} := e^{2\pi i \partial q(x,y)}$. Starting with Proposition \ref{fs-omega-form} we have:
\begin{align*}
\nu_k(\rho) =& \frac{1}{2} \theta^G_k(e) + \frac{d_\rho}{qdim(\C)} \left( \frac{\sqrt{|G|}}{2} \Theta(G,2kq) + \sum\limits_{\substack{g \in G\\1\leq i \leq N}} \langle g,g \rangle_q^k \langle x_i,x_i\rangle_{q'}^k \right)\\
 =& \frac{1}{2} \theta^G_k(e) + \frac{d_\rho \sqrt{|G|} \Theta(G, 2kq)}{|G|(2+d_\rho)} \left( \frac{1}{2} + \sum\limits_{1\leq i\leq N} \langle x_i,x_i \rangle_{q'}^k \right)\\
 =& \frac{1}{2} \theta^G_k(e) + \frac{d_\rho \sqrt{|G|} \Theta(G, 2kq)}{|G|(2+d_\rho)} \left( \frac{1}{2} + \frac{1}{2}(\Theta(G',2kq')\sqrt{|G|+4} - 1) \right)\\
 =& \frac{1}{2} \theta^G_k(e) + \frac{d_\rho \sqrt{|G|}\sqrt{|G|+4}}{2|G|(2+d_\rho)}\Theta(G, 2kq)\Theta(G, 2kq')
\end{align*}
and using the fact that $d_\rho^2 = |G| + |G|d_\rho$ we have:
\[
\frac{d_\rho \sqrt{|G|} \sqrt{|G|+4}}{2|G|(2+d_\rho)} = \frac{d_\rho(2d_\rho - |G|)}{2|G|(2+d+\rho)} = \frac{1}{2} 
\]
and then the formula for $\nu_k(\rho)$ is clear.
\end{proof}

As a corollary to the preceding theorem we obtain an easy and more natural proof of \cite[Proposition 7(b)]{eg}:

\begin{cor}\label{prod-neg}
The matrices $(S^{q,q'}, T^{q,q'})$ are modular data for a near group center only if $\Theta(G,2q)\Theta(G',2q')=-1$
\end{cor}
\begin{proof}
Suppose a near-group category  $\C= \mathcal{NG}(G,\langle,\rangle_q, e^{2\pi i q}, b, c)$  has a center $\Z(\C)$  with modular data given by $(S^{q,q'}, T^{q,q'})$. Since the simple object $\rho$ cannot contain a copy of $\mathbbm{1}$ we know that $\nu_1(\rho) = 0$. Therefore, by Theorem \ref{fs-ng2}, we must have $\Theta(G,2q)\Theta(G',2q')=-1$.
\end{proof}

Let $\left(\frac{p}{q}\right)$ be the Jacobi symbol. 

\begin{cor}
For $|G|$ odd and $k$ such that $\gcd(k, |G|\cdot|G'|)=1$ we have:
\[
\nu_k(\rho) = \frac{1}{2}\left(1 - \left(\frac{k}{|G|\cdot|G'|}\right)\right)
\]
\end{cor}

\begin{proof}
Using the decomposition into irreducible metric groups given in \cite{wall} it is easy to see that:
\[
\Theta(G,kq)\Theta(G',kq')=\left(\frac{k}{|G|\cdot|G'|}\right)\Theta(G,q)\Theta(G',q')
\]
See also \cite[\S 3 \& Lemma 3.2]{bj}.
\end{proof}

\begin{cor}
$\Theta(G',q') = -c^3$ where $(G',q')$ is the metric group associated to the center of $\mathcal{NG}(G,q,b,c)$.
\end{cor}
\begin{proof}
We've seen in Theorem \ref{iz-ng-class} that $\Theta(G,q) = \frac{1}{c^3}$, hence the above follows by the Corollary \ref{prod-neg}.
\end{proof}

The complete list of near-group categories with $K_0(\C) = K(G,|G|)$ for odd $|G|\leq13$ was obtained in \cite[Proposition 6]{eg} by finding solutions to Izumi's equations in Theorem \ref{iz-ng-class}. They also used Izumi's methods from \cite[\S 6]{iz2} and \cite{iz3} to produce the modular data of their Drinfel'd centers; see \cite[\S 4.3-4.4 \& Table 2]{eg}. Collected below are the modular data they found along with the Frobenius-Schur indicators of $\rho$ for each of these categories. Since $|G|$ is odd let $q$ be the {\it unique} quadratic form associated to the bicharacter $\langle,\rangle$ from the classification parameters.

The data uses the following notation:
\begin{itemize}
\item {\it Column 1:} $\C=\mathcal{NG}(G,q,b,c))$ with $|G|$ odd as in the above notation. (For clearer presentation, the parameters $b$ and $c$ will be given only if they are necessary to establish in-equivalence.)
\item {\it Column 2:} $(G',q')$ is the metric group from the modular data of $\Z(\C)$ from Conjecture \ref{ng-conj}. Recall $|G'| = |G| + 4$.
\item $\zeta_k = e^{\frac{2\pi i}{k}} \in \mathbb{T}$ primitive $k^{th}$ root of unity.

\end{itemize}

\medskip

\begin{center}
    \setlength{\arrayrulewidth}{0.5mm}
    \setlength{\tabcolsep}{6pt}
    \renewcommand{\arraystretch}{2.5}
    \begin{tabular}{ | c | c | c | c |}
    \hline
    $|G|=3$ & $(G',q')$ & $\nu_3(\rho)$ & $\nu_7(\rho)$ \\ \hline\hline
    $\mathcal{NG}\left(\sfrac{\mathbb{Z}}{(3)}, \frac{g^2}{3}, -, - \right)$ & $\left(\sfrac{\mathbb{Z}}{(7)}, \frac{g^2}{7}\right)$ & $\frac{3 + i\sqrt{3}}{2}$ & $\frac{1+i\sqrt{7}}{2}$ \\ \hline
    $\mathcal{NG}\left(\sfrac{\mathbb{Z}}{(3)}, \frac{-g^2}{3}, -, - \right)$ & $\left(\sfrac{\mathbb{Z}}{(7)}, \frac{-g^2}{7}\right)$ & $\frac{3 - i\sqrt{3}}{2}$ & $\frac{1 - i\sqrt{7}}{2}$\\\hline 
    \end{tabular}
\end{center}

\medskip
    
\begin{center}
    \setlength{\arrayrulewidth}{0.5mm}
    \setlength{\tabcolsep}{6pt}
    \renewcommand{\arraystretch}{2.5}
    \begin{tabular}{ | c | c | c | c | c |}
    \hline
    $|G|=5$ & $(G',q')$ & $\nu_3(\rho)$ & $\nu_5(\rho)$& $\nu_9(\rho)$ \\ \hline\hline
    $\mathcal{NG}\left(\sfrac{\mathbb{Z}}{(5)}, \frac{2g^2}{5}, -, \zeta_3 \right)$& $\left(\sfrac{\mathbb{Z}}{(9)}, \frac{2g^2}{9}\right)$ & $1+\overline{\zeta_3}$ & $\frac{5+\sqrt{5}}{2}$ & $-1$\\ \hline
    $\mathcal{NG}\left(\sfrac{\mathbb{Z}}{(5)}, \frac{2g^2}{5}, -, \overline{\zeta_3} \right)$& $\left(\sfrac{\mathbb{Z}}{(9)}, \frac{-2g^2}{9}\right)$& $1 + \zeta_3$ & $\frac{5+\sqrt{5}}{2}$ & $-1$\\ \hline
    $\mathcal{NG}\left(\sfrac{\mathbb{Z}}{(5)}, \frac{g^2}{5}, -, 1 \right)$ & $\left(\left(\sfrac{\mathbb{Z}}{(3)}\right)^2, \frac{g^2 + h^2}{3}\right)$ & $-1$ & $\frac{5-\sqrt{5}}{2}$ & $2$  \\\hline

\end{tabular}
\end{center}

\medskip

\begin{center}
    \setlength{\arrayrulewidth}{0.5mm}
    \setlength{\tabcolsep}{4pt}
    \renewcommand{\arraystretch}{2.5}
    \begin{tabular}{ | c | c | c | c |}
    \hline
    $|G|=7$ & $(G',q')$ & $\nu_7(\rho)$ & $\nu_{11}(\rho)$\\ \hline\hline
    $\mathcal{NG}\left(\sfrac{\mathbb{Z}}{(7)}, \frac{g^2}{7}, -, - \right)$& $\left(\sfrac{\mathbb{Z}}{(11)}, \frac{-2g^2}{11}\right)$ & $\frac{7-i\sqrt{7}}{2}$ & $\frac{1+i\sqrt{11}}{2}$\\ \hline
    $\mathcal{NG}\left(\sfrac{\mathbb{Z}}{(7)}, \frac{-g^2}{7}, -, - \right)$& $\left(\sfrac{\mathbb{Z}}{(11)}, \frac{2g^2}{11}\right)$ & $\frac{7+i\sqrt{7}}{2}$ & $\frac{1-i\sqrt{11}}{2}$\\ \hline
    \end{tabular}
\end{center}

\medskip

\begin{center}
    \setlength{\arrayrulewidth}{0.5mm}
    \setlength{\tabcolsep}{4pt}
    \renewcommand{\arraystretch}{2.5}
    \begin{tabular}{ | c | c | c | c | c |}
    \hline
    $|G|=9$ & $(G',q')$ & $\nu_3(\rho)$ & $\nu_9(\rho)$ & $\nu_{13}(\rho)$ \\ \hline\hline
    $\mathcal{NG}\left(\sfrac{\mathbb{Z}}{(9)}, \frac{g^2}{9}, -, - \right)$& $\left(\sfrac{\mathbb{Z}}{(13)}, \frac{-2g^2}{13}\right)$ & $1-\zeta_3$ & $3$ & $\frac{1+\sqrt{13}}{2}$  \\ \hline
    
    $\mathcal{NG}\left(\sfrac{\mathbb{Z}}{(9)}, \frac{-g^2}{9}, -, - \right)$& $\left(\sfrac{\mathbb{Z}}{(13)}, \frac{2g^2}{13}\right)$ & $1-\overline{\zeta_3}$ & $3$ & $\frac{1+\sqrt{13}}{2}$ \\ \hline
    
    $\mathcal{NG}\left((\sfrac{\mathbb{Z}}{(3)})^2, \frac{g^2-h^2}{3}, -, - \right)$& $\left(\sfrac{\mathbb{Z}}{(13)}, \frac{2g^2}{13}\right)$ & $3$ & $3$ & $\frac{1+\sqrt{13}}{2}$ \\ \hline
    \end{tabular}
\end{center}

\medskip

\begin{center}
    \setlength{\arrayrulewidth}{0.5mm}
    \setlength{\tabcolsep}{2pt}
    \renewcommand{\arraystretch}{2.5}
    \begin{tabular}{ | c | c | c | c | c | c |}
    \hline
    $|G|=11$ & $(G',q')$ & $\nu_3(\rho)$ & $\nu_5(\rho)$ & $\nu_{11}(\rho)$ & $\nu_{15}(\rho)$  \\ \hline\hline
    
    $\mathcal{NG}\left(\sfrac{\mathbb{Z}}{(11)}, \frac{g^2}{11}, -, \zeta_{12}^7 \right)$& $\left(\sfrac{\mathbb{Z}}{(15)}, \frac{2g^2}{15}\right)$ & $\frac{1-i\sqrt{3}}{2}$ & $\frac{1+\sqrt{5}}{2}$ & $\frac{11-i\sqrt{11}}{2}$ & $\frac{1+i\sqrt{15}}{2}$ \\ \hline
    
    $\mathcal{NG}\left(\sfrac{\mathbb{Z}}{(11)}, \frac{g^2}{11}, -, \overline{\zeta_{12}} \right)$& $\left(\sfrac{\mathbb{Z}}{(15)}, \frac{g^2}{15}\right)$ & $\frac{1+i\sqrt{3}}{2}$ & $\frac{1-\sqrt{5}}{2}$ & $\frac{11-i\sqrt{11}}{2}$ & $\frac{1+i\sqrt{15}}{2}$ \\ \hline
    
    $\mathcal{NG}\left(\sfrac{\mathbb{Z}}{(11)}, \frac{-g^2}{11}, -, \zeta_{12} \right)$& $\left(\sfrac{\mathbb{Z}}{(15)}, \frac{-g^2}{15}\right)$ & $\frac{1-i\sqrt{3}}{2}$ & $\frac{1-\sqrt{5}}{2}$ & $\frac{11+i\sqrt{11}}{2}$ & $\frac{1-i\sqrt{15}}{2}$ \\ \hline
    
    $\mathcal{NG}\left(\sfrac{\mathbb{Z}}{(11)}, \frac{-g^2}{11}, -, \zeta_{12}^5 \right)$& $\left(\sfrac{\mathbb{Z}}{(15)}, \frac{-2g^2}{15}\right)$ & $\frac{1+i\sqrt{3}}{2}$ & $\frac{1+\sqrt{5}}{2}$ & $\frac{11+i\sqrt{11}}{2}$ & $\frac{1-i\sqrt{15}}{2}$ \\ \hline
    \end{tabular}
\end{center}

\medskip

\begin{center}
    \setlength{\arrayrulewidth}{0.5mm}
    \setlength{\tabcolsep}{4pt}
    \renewcommand{\arraystretch}{2.5}
    \begin{tabular}{ | c | c | c | c |}
    \hline
    $|G|=13$ & $(G',q')$ & $\nu_{13}(\rho)$ & $\nu_{17}(\rho)$ \\ \hline\hline
    
    $\mathcal{NG}\left(\sfrac{\mathbb{Z}}{(13)}, \frac{g^2}{13}, b_1 , -1 \right)$ & $\left(\sfrac{\mathbb{Z}}{(17)}, \frac{3g^2}{17}\right)$  & $\frac{13-\sqrt{13}}{2}$ & $\frac{1+\sqrt{17}}{2}$ \\ \hline
    
    $\mathcal{NG}\left(\sfrac{\mathbb{Z}}{(13)}, \frac{g^2}{13}, b_2 , -1 \right)$ & $\left(\sfrac{\mathbb{Z}}{(17)}, \frac{3g^2}{17}\right)$  & $\frac{13-\sqrt{13}}{2}$ & $\frac{1+\sqrt{17}}{2}$ \\ \hline
    
    $\mathcal{NG}\left(\sfrac{\mathbb{Z}}{(13)}, \frac{2g^2}{13}, b_3 , 1 \right)$ & $\left(\sfrac{\mathbb{Z}}{(17)}, \frac{g^2}{17}\right)$  & $\frac{13+\sqrt{13}}{2}$ & $\frac{1-\sqrt{17}}{2}$ \\ \hline
    
    $\mathcal{NG}\left(\sfrac{\mathbb{Z}}{(13)}, \frac{2g^2}{13}, b_4 , 1 \right)$ & $\left(\sfrac{\mathbb{Z}}{(17)}, \frac{g^2}{15}\right)$  & $\frac{13+\sqrt{13}}{2}$ & $\frac{1-\sqrt{17}}{2}$ \\ \hline
    \end{tabular}
\end{center}

\begin{rem}
See that for $G= \mathbb{Z}/(13)$ we have two pairs of inequivalent fusion categories with the same indicators, hence the near group fusion ring $NG(\mathbb{Z}/(13), 13)$ does not have FS indicator rigidity. Note that the lesser odd order groups \emph{do} exhibit FS indicator rigidity.
\end{rem}

\section{Frobenius-Schur indicators for Haagerup-Izumi fusion categories}

Near groups are examples of {\bf quadratic} fusion categories: those tensor-generated by a single non-invertible simple object $\rho$ where the set of simple objects is given by:
\[
G \cup \{ \bar{g} \otimes \rho \, | \, \bar{g} \textrm{ a coset representative in } G/H \}
\]
where $H$ is some subgroup of $G$. Near groups correspond to $H=G$. On the other end of the spectrum, the Haagerup-Izumi fusion categories correspond to $H=\{e\}$.

\begin{defn}
$\C$ is a {\bf Haagerup-Izumi} fusion category if:
\[
K_0(\C) = HI(G) := \mathbb{Z}[G \cup \{ g\rho \, | \, g\in G\}]
\]
where multiplication is given by the group law and:
\[
g(h\rho) = (gh)\rho = (h\rho)g^{-1}
\]
\[
(g\rho)(h\rho) = gh^{-1} + \sum\limits_{a\in G} g\rho
\]
\end{defn}

\subsection{Classification and examples}
The complete lists of Haagerup-Izumi categories for $G=\mathbb{Z}/(3)$ and $G=\mathbb{Z}/(5)$ were found by Evans-Gannon without assuming unitarity in \cite{eg2} by generalizing Izumi's methods to endomorphisms of Leavitt algebras. These categories are classified up to isomorphism of the group $G$ and the following parameters:
\begin{itemize}
\item a sign $\pm$,  
\item $\omega$ a third root of unity, and
\item $A\in M_{|G|}(\mathbb{C})$ a complex matrix
\end{itemize}
all satisfying some relations given in \cite[Theorem 1]{eg2}. An Haagerup-Izumi category with the above parameters will be denoted:
\[
\mathcal{HI}(G,\pm,\omega,A)
\]
The notion of equivalence for the parameters is given in \cite[Theorem 2(b)]{eg2}. In particular, the category is unitary if and only if both the sign is $+$ and $A$ is hermitian \cite[Theorem 2(c)]{eg2}.

The most important examples of HI categories are the Yang-Lee system of sectors, which is the unique non-unitary such category with $G$ the trivial group, and the system of sectors for the Haagerup subfactor, which is a unitary HI category with $G=\mathbb{Z}/(3)$.

\subsection{Indicators when $|G|$ is odd}

When $\C$ is a Haagerup-Izumi fusion category the modular data for the center $\Z(\C)$ was computed by Evans--Gannon in \cite[\S 6.3]{eg2} and is given as follows:

\begin{center}
    \setlength{\arrayrulewidth}{0.5mm}
    \setlength{\tabcolsep}{8pt}
    \renewcommand{\arraystretch}{2}
    \begin{tabular}{ | c | c | c | c |}
    \hline\hline
    $X \in \Irr(\Z(\C))$ & $F(X)$ &  parameter &  $\theta_X$\\ \hline\hline
    $\mathbbm{1}$ & $\mathbbm{1}$ & 1 & 1 \\ \hline
    $B$ & $\mathbbm{1} + \sum\limits_{g\in G}g\otimes\rho$ & 1 & 1 \\ \hline
    $A_\psi = A_{\bar{\psi}}$ & $2\mathbbm{1} + \sum\limits_{g\in G} g \otimes \rho$  & $\psi\in \widehat{G}$  &  1 \\
    \hline
    $C^{(h)}_\phi = C^{(h^{-1})}_\phi$ ($h\in G$) & $h + h^{-1} + \sum\limits_{g \in G} g \otimes \rho$ & $\phi\in\widehat{G}$ & $\phi(g)$ \\ \hline
    $D_j$ & $\sum\limits_{g\in G} g\otimes \rho$ & $1\leq j \leq \frac{|G|^2 + 3}{2}$ & $\zeta_j$\\ \hline
    \end{tabular}
\end{center}

\medskip

\noindent The $\zeta_j$ are a solutions to a system of equations with coefficients given by $\pm, \omega, A$. These equations are $(6.14)$, $(6.16)-(6.19)$ in \cite[\S 6.2]{eg2}. See \cite[Proposition 2]{eg2}. For $G$ odd order they make another conjecture:

\begin{conj}\label{hi-conj}\cite[Conjecture 1]{eg2}
Suppose $|G|$ is odd. Then there exists a metric group $(H,q'')$ of order $|G|^2+4=2m+1$ such that the simple objects $D_j$ are indexed by $h \in H\backslash\{e\}$ where $D_h = D_{h^{-1}}$ and:
\[
\theta_{D_h} = e^{2\pi i m q''(h)}
\]
\end{conj}

\begin{rem}
They show in \cite[Theorem 3]{eg2} that the conjecture is true for Haagerup-Izumi fusion categories with $|G|= 1, 3, 5$.
\end{rem}

\begin{thm}\label{hi-fs}
Suppose $\C$ is a Haagerup-Izumi fusion category with $|G|$ odd satisfying Conjecture \ref{hi-conj}. Then:
\[
\nu_k(\rho) = \frac{1}{2}\theta^G_k(e) + \frac{1}{2}\Theta(H,kmq'')
\]
\end{thm}

\begin{proof}
Let $d = qdim(\rho)$ in the category $\C$. Again by using Theorem \ref{fs-sph};:
\begin{align*}
\nu_k(\rho) &= \frac{1}{qdim(\C)} \bigg( qdim(B) +  \sum\limits_{\bar{\psi} \neq \psi \in \widehat{G}} qdim(A_\psi)\\ &+ \sum\limits_{e \neq h^{-1} \neq h \in G} \sum\limits_{\phi\in\widehat{G}} \theta^k_{C^h_\phi} qdim(C^h_\phi) + \sum\limits_{\gamma^{-1} \neq \gamma\in H} \theta_{D_\gamma} qdim(D_\gamma) \bigg)
\end{align*}
Letting $|G| = 2n + 1$ and $|H| = 2m + 1$ we may enumerate these odd order groups as: 
\[
G = \{ e,\, g_i,\, g_i^{-1} \, | \, 1\leq i \leq n \} \quad \textrm{and}\quad H = \{ e,\, h_j,\, h_j^{-1} \, | \, 1\leq j \leq m \}
\]
this gives us:
\begin{align*}
\nu_k(\rho) &= \frac{1}{qdim(\C)} \bigg( |G| + |G|d + |G|dn + (2+|G|d)\sum\limits_{g_i, \phi} \phi(g_i)^k + |G|d \sum\limits_{h_j \in H} \zeta_j^k \bigg)\\
\end{align*}
By the same argument as in the proofs of Theorems \ref{fs-ng1} and \ref{fs-ng2} we can see:
\[
\sum\limits_{g_i, \phi} \phi(g_i)^k = \frac{|G|}{2}(\theta^G_k(e) - 1)
\]
Hence using the expression for the center's ribbon structure from Conjecture \ref{hi-conj} and the fact that $|H| = |G|^2 + 4$ and $qdim(\C) = 2|G| + d|G|^2$ we see:
\begin{align*}
\nu_k(\rho) &= \frac{|G|}{qdim(\C)} \bigg( \frac{2 + |G|d}{2} \theta^G_k(e) + d + dn - \frac{|G|d}{2} + d\sum\limits_{h_j \in H} e^{2\pi i k m q''(h_j)} \bigg)\\
&= \frac{1}{2} \theta^G_k(e) + \frac{|G|}{2qdim(\C)} \bigg(2d + 2dn - |G|d + d(\sqrt{|H|} \Theta(H, kmq'') - 1) \bigg)\\
&= \frac{1}{2} \theta^G_k(e) + \frac{|G|d \sqrt{|G|^2 + 4}}{2 qdim(\C)} \Theta(H, kmq'')\\
&= \frac{1}{2} \theta^G_k(e) + \frac{1}{2}\Theta(H, kmq'')
\end{align*}
\end{proof}

Now we collect in the tables below the values of the Frobenius-Schur indicators for the Haagerup-Izumi categories constructed in \cite{eg3}:

\medskip

\begin{center}
    \setlength{\arrayrulewidth}{0.5mm}
    \setlength{\tabcolsep}{4pt}
    \renewcommand{\arraystretch}{2.5}
    \begin{tabular}{ | c | c | c | c | c |}
    \hline
    $G=\mathbb{Z}/(3)$ & $(H,q'')$ & $\nu_k(\rho)$ & $\nu_3(\rho)$ & $\nu_{13}(\rho)$ \\ \hline\hline
    $\mathcal{HI}\left(\sfrac{\mathbb{Z}}{(3)}, + , 1 , A_1 \right)$& $\left(\sfrac{\mathbb{Z}}{(13)}, \frac{g^2}{13}\right)$ & $\frac{1}{2}\left( 1 - \left(\frac{k}{13}\right) \right)$ & $1$ & $\frac{1+\sqrt{13}}{2}$  \\ \hline
    
    $\mathcal{HI}\left(\sfrac{\mathbb{Z}}{(3)}, + , 1 , A_2 \right)$& $\left(\sfrac{\mathbb{Z}}{(13)}, \frac{g^2}{13}\right)$ &  $\frac{1}{2}\left( 1 - \left(\frac{k}{13}\right) \right)$ & $1$ & $\frac{1+\sqrt{13}}{2}$ \\ \hline
    
    $\mathcal{HI}\left(\sfrac{\mathbb{Z}}{(3)}, -, 1, A_3 \right)$& $\left(\sfrac{\mathbb{Z}}{(13)}, \frac{2g^2}{13}\right)$ & $\frac{1}{2}\left( 1 + \left(\frac{k}{13}\right) \right)$ & $2$ & $\frac{1+\sqrt{13}}{2}$\\ \hline
    
    $\mathcal{HI}\left(\sfrac{\mathbb{Z}}{(3)}, -, 1, A_4 \right)$& $\left(\sfrac{\mathbb{Z}}{(13)}, \frac{2g^2}{13}\right)$ & $\frac{1}{2}\left( 1 + \left(\frac{k}{13}\right) \right)$ & $2$ & $\frac{1+\sqrt{13}}{2}$\\ \hline
    
    \end{tabular}
\end{center}

\medskip

\begin{center}
    \setlength{\arrayrulewidth}{0.5mm}
    \setlength{\tabcolsep}{4pt}
    \renewcommand{\arraystretch}{2.5}
    \begin{tabular}{ | c | c | c | c | c |}
    \hline
    $G=\mathbb{Z}/(5)$ & $(H,q'')$ & $\nu_k(\rho)$ & $\nu_5(\rho)$ & $\nu_{29}(\rho)$ \\ \hline\hline
    
    $\mathcal{HI}\left(\sfrac{\mathbb{Z}}{(5)}, + , 1 , A_6 \right)$& $\left(\sfrac{\mathbb{Z}}{(13)}, \frac{g^2}{29}\right)$ & $\frac{1}{2}\left( 1 - \left(\frac{k}{13}\right) \right)$ & $2$ & $\frac{1+\sqrt{29}}{2}$  \\ \hline
    
    $\mathcal{HI}\left(\sfrac{\mathbb{Z}}{(5)}, + , 1 , A_7 \right)$& $\left(\sfrac{\mathbb{Z}}{(13)}, \frac{g^2}{29}\right)$ &  $\frac{1}{2}\left( 1 - \left(\frac{k}{13}\right) \right)$ & $2$ & $\frac{1+\sqrt{29}}{2}$ \\ \hline
    
    $\mathcal{HI}\left(\sfrac{\mathbb{Z}}{(5)}, -, 1, A_8 \right)$& $\left(\sfrac{\mathbb{Z}}{(13)}, \frac{2g^2}{29}\right)$ & $\frac{1}{2}\left( 1 + \left(\frac{k}{13}\right) \right)$ & $3$ & $\frac{1+\sqrt{29}}{2}$\\ \hline
    
    $\mathcal{HI}\left(\sfrac{\mathbb{Z}}{(5)}, -, 1, A_9 \right)$& $\left(\sfrac{\mathbb{Z}}{(13)}, \frac{2g^2}{29}\right)$ & $\frac{1}{2}\left( 1 + \left(\frac{k}{13}\right) \right)$ & $3$ & $\frac{1+\sqrt{29}}{2}$\\ \hline
    
    \end{tabular}
\end{center}

\begin{rem}
See that for each of $\mathbb{Z}/(3)$ and $\mathbb{Z}/(5)$ we have two pairs of inequivalent fusion categories with the same indicators, hence the Haagerup-Izumi fusion rings do not have FS rigidity. 
\end{rem}
Note that in this case as well as the $m=|G|=13$ near-group case the pairs have centers with the same modular data (although it is not established whether the centers are equivalent). In view of this we formulate the following conjecture:
\begin{conj}
Two fusion categories with a given Grothendieck ring that are also Morita equivalent cannot be distinguished by their Frobenius-Schur indicators.
\end{conj}


\end{document}